\documentclass[sn-mathphys]{sn-jnl}
\usepackage[utf8]{inputenc} 
\usepackage{float}
\usepackage{enumerate}
\usepackage{amsmath}
\usepackage{amssymb}
\usepackage{amsthm}
\usepackage{marvosym}
\usepackage{xcolor}
\usepackage{comment}
\usepackage{calc}
\usepackage{manfnt}                      
\usepackage{mathtools} 
\usepackage[activate={true,nocompatibility},
            final,
            tracking=true,
            kerning=true,
            spacing=true]{microtype}
\usepackage{interval}
\usepackage{colortbl}
\usepackage{hyperref}

\jyear{2022}%

\theoremstyle{thmstyleone}%
\newtheorem{theorem}{Theorem}[section]
\newtheorem{lemma}[theorem]{Lemma}
\newtheorem{corollary}[theorem]{Corollary}

\theoremstyle{thmstyletwo}%
\newtheorem{example}[theorem]{Example}%

\theoremstyle{thmstylethree}%

\newcommand{\RR}{\mathbb{R}}

\newcommand{\NN}{\mathbb{N}}

\newcommand{\scp}[2]{\langle{#1},{#2}\rangle}
\newcommand{\vertiii}{{\vert\kern-0.25ex\vert\kern-0.25ex\vert}} 

\renewcommand{\phi}{\varphi}

\newcommand{\bigO}{\mathcal{O}}

\DeclareMathOperator{\prox}{prox}
\DeclareMathOperator{\proj}{proj}

\DeclarePairedDelimiter{\norm}{\lVert}{\rVert}
\DeclarePairedDelimiter{\abs}{\lvert}{\rvert}

\begin{document}

\title[Article Title]{Chambolle-Pock's Primal-Dual Method with Mismatched Adjoint}

\author*[1]{\fnm{Dirk A.} \sur{Lorenz}}\email{d.lorenz@tu-braunschweig.de}
\author*[1]{\fnm{Felix} \sur{Schneppe}}\email{f.schneppe@tu-braunschweig.de}


\affil*[1]{\orgdiv{Institut f\"{u}r Analysis und Algebra}, \orgname{TU Braunschweig}, \orgaddress{\street{Universit\"{a}tsplatz 2}, \city{Braunschweig}, \postcode{38106}, \state{Niedersachsen}, \country{Deutschland}}}


\abstract{
  The primal-dual method of Chambolle and Pock is a widely used algorithm to solve various optimization problems written as convex-concave saddle point problems.
Each update step involves the application of both the forward linear operator and its adjoint. However, in practical applications like computerized tomography, it is often computationally favourable to replace the adjoint operator by a computationally more efficient approximation. This leads to an adjoint mismatch in the algorithm.

In this paper, we analyze the convergence of Chambolle-Pock's primal-dual method under the presence of a mismatched adjoint in the strongly convex setting. We present an upper bound on the error of the primal solution and derive stepsizes and mild conditions under which convergence to a fixed point is still guaranteed. Furthermore we show linear convergence similar to the result of Chambolle-Pock's primal-dual method without the adjoint mismatch. Moreover, we illustrate our results both for an academic and a real-world inspired application.
}


\maketitle

\label{sect:intro}

\section{Introduction}
Inverse problems occur whenever unknown quantities are measured indirectly and in many cases the measurement process introduces measurement noise. Nevertheless, these inverse problems appear in many practical applications and are often approached by 
solving minimization problems, which can be formulated as the minimization of an expression
\begin{align}\label{eq:minFA+G}
\min_{x \in X} F(Ax) + G(x)
\end{align}
on a Hilbert space $X$ with a linear and bounded operator $A: X \to Y$. Sometimes these kind of problems are hard to solve and it can be beneficial to examine the equivalent dual problem
\begin{align}\label{eq:maxDualFA+G}
\min_{y \in Y} G^*(-A^*y) + F^*(y)
\end{align}
on the Hilbert space $Y$ or the saddle point problem 
\begin{align*}
\min_{x \in X} \max_{y \in Y} G(x) + \langle Ax,y \rangle - F^*(y).
\end{align*}
with the Fenchel conjugates $G^*, F^*$ of $G,F$ instead of the primal problem above.
If both $G: X \to \overline{\RR}$ and $F^*: Y \to \overline{\RR}$ are proper, convex, lower
semicontinuous functionals defined on Hilbert spaces $X,Y$, the primal-dual algorithm of Chambolle and Pock~\cite{chambolle-pock-first-order}, defined as
\begin{align}\label{eq:cp}
  \begin{split}
x^{i+1} &= \operatorname{prox}_{\tau_i G}(x^i - \tau_i A^* y^i), \\
\bar{x}^{i+1} &= x^{i+1} + \omega_i (x^{i+1} - x^i), \\
y^{i+1} &= \operatorname{prox}_{\sigma_{i+1} F^*}(y^i + \sigma_{i+1} A \bar{x}^{i+1}),
\end{split}
\end{align}
for all $i \in \NN$, with positive stepsizes $\tau_i$ and $\sigma_i$ and extrapolation parameter $\omega_i$ has proven to be a simple and effective solution method. Using constant stepsizes, this method converges weakly, if $F^*$ and $G$ are convex and lower semi continuous functionals. Furthermore linear convergence is proven in~\cite{chambolle-pock-first-order}, if both functionals are strongly convex. If only one of the functionals is strongly convex, while the other is convex, an accelerated version of the algorithm with varying stepsizes is proven in~\cite{chambolle-pock-first-order} to converge with rate of $\norm{x^{N}-\hat x}^{2} = \bigO{(N^{-2})}$ where $\hat x$ is a solution of~\eqref{eq:minFA+G}. 

In practical applications, however, it can happen that the operator and its adjoint are given as two seperate implementations of discretizations of a continuous operator and its adjoint. If the implementations use the ``first dualize, then discretize'' approach, it may happen, that the discretizatons are not adjoint to each other. Sometimes, this even happens on purpose, for example to save computational time or to impose certain structure for the image of the adjoint operator  \cite{Buffiere2010InSE, Riddell1995TheAI,zeng2000unmatched}. The influence of such a mismatch has been studied for various algorithms. \cite{Savanier2021ProximalGA,Chouzenout2021pgm-adjoint,zeng2000unmatched,Lorenz2018TheRK, Dong2019FixingNO, Elfving2018UnmatchedPP}

In this paper we examine the convergence of the Chambolle-Pock method in the case of a mismatched adjoint, i.e., we examine the algorithm 
\begin{align}\label{eq:cpmm}
  \begin{split}
x^{i+1} &= \operatorname{prox}_{\tau_i G}(x^i - \tau_i V^* y^i), \\
\bar{x}^{i+1} &= x^{i+1} + \omega_i (x^{i+1} - x^i), \\
y^{i+1} &= \operatorname{prox}_{\sigma_{i+1} F^*}(y^i + \sigma_{i+1} A \bar{x}^{i+1}).
\end{split}
\end{align}
with a linear operator $V: X \to Y$ instead of $A$ for convergence to a fixed point of (\ref{eq:cpmm}) in the case where both $G$ and $F^{*}$ are strongly convex.

\begin{example}[Counterexample for convergence]
  Here is a simple example that shows that the mismatched iteration does not necessarily converge. We consider the problem $\min_{x} \norm{x}_{1}$ on $\RR^{n}$, which is of the form (\ref{eq:minFA+G}), and model this with $A = I$, $F(y) = \norm{y}_1$, and $G\equiv 0$. We consider the most basic form of Chambolle-Pock's method with constant $\tau,\sigma > 0$ and $\omega=1$, i.e. the mismatched iteration is  
\begin{align*}
  x^{i+1} &= x^{i} - \tau V^{T}y^{i}\\
  y^{i+1} & = \proj_{[-1,1]}(y^{i} + \sigma A(2x^{i+1}-x^{i})).
\end{align*}
If we consider the mismatch $V = -\alpha I$ with $\alpha>0$ (instead of $I$), the iteration becomes 
\begin{align*}
  x^{i+1} &= x^{i} + \alpha\tau y^{i}\\
  y^{i+1} & = \proj_{[-1,1]}(y^{i} + \sigma(2x^{i+1}-x^{i})).
\end{align*}
If we initialize with $x^{0} > 0$ and $y^{0} > 0$ (component-wise), we get that the entries in $x^{i}$ are strictly increasing and hence, will not converge to the unique solution $x=0$.

Note that $(x,y) = (0,0)$ is both a saddle point and a fixed point of the mismatched iterations in this case.

\end{example}

Before we analyze the convergence of the mismatched iteration~\eqref{eq:cpmm} we provide a result that shows that fixed points of~\eqref{eq:cpmm} are close to the true solution of~\eqref{eq:cp} if the norm $\norm{A-V}$ is small.

\begin{theorem}\label{thm:error-estimate}
If $G$ is a $\gamma_G$-strongly convex function, $(x^*, y^*)$ is the fixed point of the original Chambolle-Pock method (\ref{eq:cp}) and $(\hat{x},\hat{y})$ is the fixed point of the Chambolle-Pock method with mismatched adjoint (\ref{eq:cpmm}), it holds that
\[
  \norm{x^* - \hat{x}} \leq \frac{1}{\gamma_G} \norm{(V-A)^* \hat{y}}.
\]
\end{theorem}
\begin{proof}
Since $\partial G$ is $\gamma_{G}$-strongly monotone and $\partial F^*$ is monotone, we can conclude for
$-A^* y^* \in \partial G(x^*), -V^*\hat{y} \in \partial G(\hat{x}),
A x^* \in \partial F^*(y^*), A \hat{x} \in \partial F^*(\hat{y})$ that
\[\scp{x^*-\hat{x}}{-A^* y^* + V^* \hat{y}} \geq  \gamma_G \norm{x^* - \hat{x}}^2\]
and 
\[\scp{x^*-\hat{x}}{A^* (y^* - \hat{y})} = \scp{A(x^*-\hat{x})}{y^* - \hat{y}} \geq 0.\]
These sum up to
\[\scp{x^*-\hat{x}}{(V-A)^* \hat{y}} \geq \gamma_G \norm{x^* - \hat{x}}^2 .\]
Furthermore, it is
\[\gamma_G \norm{x^* - \hat{x}}^2 \leq \scp{x^*-\hat{x}}{(V-A)^* \hat{y}} \leq \norm{x^* - \hat{x}} \norm{(V-A)^* \hat{y}},\]
which shows
\[\gamma_G \|x^* - \hat{x}\| \leq \| (V-A)^* \hat{y} \|.\]
\end{proof}

Notably, we cannot show that the mismatched algorithm will converge to the original solution (a situation which possible for other mismatched iterations~\cite{Lorenz2018TheRK}). However, since we can bound the difference of fixed points up to a multiplicative constant by the norm of the difference between the correct and the mismatched adjoint, the analysis of the Chambolle-Pock method with mismatched adjoint can be of interest in practical applications. Moreover, in applications in computerized tomography, mismatched adjoints are also used on purpose~\cite{Lorenz2018TheRK,Savanier2021ProximalGA,Chouzenout2021pgm-adjoint,zeng2000unmatched}.

The rest of the paper is structured as follows. In Section~\ref{sec:conv} we reformulate (\ref{eq:cpmm}), introduce the concept of test operators from~\cite{valkonen_test} and provide some technical lemmas that we need to prove convergence of the mismatched iteration in Section~\ref{sec:convergence}. In Section~\ref{sect:examples}, we present numerical examples and Section~\ref{sect:conclusion} concludes the paper.

Throughout this paper, we will use $\scp{x}{x'}_T := \scp{Tx}{x'}$ and the seminorm $\norm{x}_T ^{2}:= \scp{x}{x}_T$ for a (not necessarily symmetric) positive semidefinite operator $T$. \footnote{In the case that $T$ is not positive semidefinite, the expression $\norm{x}_{T}$ will not be used, however, its square $\norm{x}_{T}^{2}$ will (but may be a negative number.)} In case of having $T = I$, we will denote the Hilbert space norm as $\norm{x}$ without the subscripts. Additionally, we will use the notation $\mathcal{L}(\mathcal{U},\mathcal{V})$ for the space of bounded linear operators $L: \mathcal{U} \to \mathcal{V}$ between Hilbert spaces $\mathcal{U}$ and $\mathcal{V}$ with the corresponding operator norm $\norm{L} = \inf\{c \geq 0: \norm{Lx} \leq c \norm{x} \text{ for all } x \in \mathcal{U}\}$. Furthermore, we will write $A \geq B$ for operators $A,B \in \mathcal{L}(\mathcal{U},\mathcal{U})$, if $A-B$ is positive semidefinite.

\section{Preliminaries}
\label{sec:conv}
In this section we present the reformulation of the mismatched iteration as a preconditioned proximal point method, recall the results from~\cite{valkonen_test} on which our analysis relies and provide the necessary technical estimates needed for the convergence proof. Note that we do not prove the existence of a fixed point of the mismatched iteration, but take its existence for granted. 

\subsection{Subdifferential Reformulation}
Since the original proof of Chambolle and Pock in \cite{chambolle-pock-first-order} relies on having the exact adjoint, we use a different approach, namely the reformulation of the method as a preconditioned proximal point method from~\cite{he-yuan-pdhg} (see also~\cite{valkonen_acc}). Recall that the proximal operator of a proper, convex and lower semicontinuous functional $f: X \to \overline{\RR}$ is defined as $\prox_{f}(x) = \arg \min_{y \in X} f(y) + \tfrac12\norm{y-x}^{2}$ and that it holds 
\begin{align*}
v = \prox_f(x) \Leftrightarrow x-v \in \partial f(v).
\end{align*}
One sees that each stationary point $\hat u =
\begin{pmatrix}
  \hat x\\\hat y
\end{pmatrix}$ of iteration~(\ref{eq:cpmm})
fulfills $0 \in H\hat{u}$ with
\begin{align}\label{eq:def_H}
H(u) = H(x,y) = \begin{pmatrix} \partial G(x) + V^* y \\ \partial F^*(y) - Ax \end{pmatrix}.
\end{align}
With $\bar{x}^{i+1} = (1+\omega_i) x^{i+1} - \omega_i x^i$, we rewrite the update steps of algorithm~(\ref{eq:cpmm}) as
\begin{equation}\label{eq:reform-cpmm}
0 \in  \left( \begin{array}{c} x^{i+1} + \tau_i \partial G(x^{i+1}) - x^i + \tau_i V^* y^i \\
y^{i+1} + \sigma_{i+1} \partial F^*(y^{i+1})- y^i - \sigma_{i+1} A \bar{x}^{i+1} \end{array} \right).
\end{equation}
We rewrite this inclusion with the help of the operator $\tilde H_{i+1}$ defined by
\begin{align}\label{eq:def-tildeH}
\tilde{H}_{i+1}(u) = \begin{pmatrix} \partial G(x) + V^*y \\ \partial F^*(y) - A [(1+\omega_i) x - \omega_i x^i ] - \omega_i V (x^i -x) \end{pmatrix},
\end{align}
the  \textit{preconditioner}
\begin{align}\label{eq:def-precond}
M_{i+1} = \begin{pmatrix} I & - \tau_i V^* \\ - \omega_i \sigma_{i+1} V & I\end{pmatrix},
\end{align}
and the \textit{step length operator}
\begin{align}\label{eq:def-steplength}
W_{i+1} = \begin{pmatrix} \tau_i I & 0 \\ 0 & \sigma_{i+1} I \end{pmatrix}.
\end{align}
as the following preconditioned proximal point iteration
\begin{align}\label{eq:HM-iteration}
0 \in W_{i+1} \tilde{H}_{i+1}(u^{i+1}) + M_{i+1} (u^{i+1} - u^i)
\end{align}
which is exactly the mismatched iteration~(\ref{eq:reform-cpmm}).

For this formulation of the iteration we can apply results from \cite{valkonen_test} and \cite{valkonen_acc}. We quote this (slightly modified) general theorem of \cite[Theorem 2.1]{valkonen_acc}:

\begin{theorem} \label{valkonen_theorem}
Let $\mathcal{U}$ be a Hilbert space,  $\tilde{H}_{i+1}: \mathcal{U} \rightrightarrows \mathcal{U}$ and let 
    $M_{i+1}, W_{i+1}, Z_{i+1} \in \mathcal{L}(\mathcal{U},\mathcal{U})$ for $i \in \mathbb{N}$. Suppose that
    \begin{equation}
	\label{eqn:firstcond}
	0 \in W_{i+1} \tilde{H}_{i+1}(u^{i+1}) + M_{i+1}(u^{i+1}-u^i)
	\end{equation}
    is solvable for $\{u^{i+1}\}_{i \in \mathbb{N}} \subset \mathcal{U}$ and let $\hat{u} \in \mathcal{U}$ be a stationary point of the iteration.
    
    If $Z_{i+1}M_{i+1}$ is self-adjoint and for some $\Delta_{i+1} \in \mathbb{R}$ the condition
    \begin{align}
	\label{eqn:secondcond}
	\scp{\tilde{H}_{i+1}(u^{i+1})}{u^{i+1} - \hat{u}}_{Z_{i+1} W_{i+1}} &\nonumber \\ \geq  \frac12 \norm{u^{i+1} &- \hat{u}}_{Z_{i+2}M_{i+2} - Z_{i+1}M_{i+1}}^2 - \frac12 \norm{u^{i+1}-u^i}^2_{Z_{i+1}M_{i+1}} - \Delta_{i+1}	
	\end{align}
   holds for all $i \in \mathbb{N}$, then so does the descent inequality
   \[
     \frac12 \norm{u^N - \hat{u}}^2_{Z_{N+1}M_{N+1}} \leq \frac12 \|u^0-\hat{u}\|^2_{Z_1M_1} + \sum_{i=0}^{N-1} \Delta_{i+1}.
   \]
\end{theorem}

The maps $\tilde{H}_{i}$, $M_{i}$, and $W_{i}$ from \eqref{eq:def-tildeH},~\eqref{eq:def-precond}, and~\eqref{eq:def-steplength} are defined on $\mathcal{U} = X \times Y$. Then the iteration~(\ref{eqn:firstcond}) is exactly the iteration~\eqref{eq:cpmm}. The operator $Z_{i}$ and the number $\Delta_{i}$ are yet to be defined and are used to establish inequality~(\ref{eqn:secondcond}) and the final descent inequality. The operator $Z_{i}$ is called \emph{test operator} and the $\Delta_{i}$ can be used to further quantify the descent. 

We will introduce the test operator $Z_{i}$ and the quantities $\Delta_{i}$ in the next subsection and aim for  non-positive $\left( \Delta_{i}\right)_{i \in \NN}$, but also want that the operators $Z_{i+1}M_{i+1}$ grow as fast as possible to obtain fast convergence.

Consequently, our next aim is to show that
\begin{itemize}
\item with the right step length choices an operator $Z_{i+1}$ with $Z_{i+1} M_{i+1}$ being self-adjoint exists, and 
\item for some non-positive $\Delta_{i+1}$ the inequality (\ref{eqn:secondcond}) 
can be obtained.
\end{itemize}

\subsection{Test operator and step length bounds}
We choose the test operator as 
\begin{align}\label{eq:def-test-op}
Z_{i+1} := \begin{pmatrix} \phi_i I & 0 \\ 0 & \psi_{i+1}I \end{pmatrix}
\end{align}
for $i \in \NN$ and show that with this choice we can fulfill the assumptions in Theorem~\ref{valkonen_theorem} for appropriate $\phi_{i}$ and $\psi_{i+1}$.
First, we need that $Z_{i+1}M_{i+1}$ is self adjoint. Since this operator is
\begin{align}\label{eq:ZM}
  Z_{i+1}M_{i+1} =
  \begin{pmatrix}
    \phi_{i}I & -\tau_{i}\phi_{i}V^{*}\\
    -\omega_{i}\sigma_{i+1}\psi_{i+1}V & \psi_{i+1}I
  \end{pmatrix},
\end{align}
we assume that the values $\phi_{i},\psi_{i+1}$ of the test operator fulfill
\begin{align}\label{eq:adjointpd}
  \omega_{i}\sigma_{i+1}\psi_{i+1} = \tau_{i}\phi_{i}.
\end{align}
Next we introduce the ``tested dual stepsize'' $\psi_{i}\sigma_{i}$
\begin{align}\label{eq:def-eta}
  \eta_{i} = \psi_{i}\sigma_{i}
\end{align}
and define the extrapolation constant $\omega_{i}$ as
\begin{align}\label{eq:def-omega}
  \omega_{i} = \tfrac{\eta_{i}}{\eta_{i+1}}.
\end{align}
Consequently, the ``tested primal stepsize'' $\tau_{i}\phi_{i}$ also fulfills
\begin{align}\label{eq:tested-primal-stepsize}
  \eta_{i} = \eta_{i+1}\omega_{i} = \omega_{i}\sigma_{i+1}\psi_{i+1} = \tau_{i}\phi_{i}.
\end{align}
Furthermore, Theorem \ref{valkonen_theorem} needs $Z_{i+1} M_{i+1}$ to be positive semidefinite which we show now.
\begin{lemma}[\cite{valkonen_acc}, Lemma 3.4]\label{lem:lowerbound-ZM}
Let $i \in \NN$ and suppose that conditions~(\ref{eq:adjointpd}),~(\ref{eq:def-eta}) and~(\ref{eq:def-omega}) hold. Then we have that $Z_{i+1} M_{i+1}$ is self-adjoint and satisfies 
\begin{align}\label{eq:estimate-ZM}
Z_{i+1} M_{i+1}  = \begin{pmatrix} \phi_i I & - \eta_i V^* \\ - \eta_i V & \psi_{i+1} I \end{pmatrix} \geq\begin{pmatrix} \delta \phi_i I & 0 \\ 0 & \psi_{i+1} I - \frac{\eta_i^2 }{\phi_i (1-\kappa)} VV^* \end{pmatrix}
\end{align}
for all $\delta \in [0,\kappa]$ with $\kappa \in (0,1)$.
\end{lemma}
\begin{proof}
  First note that from~\eqref{eq:ZM},~(\ref{eq:adjointpd}),~\eqref{eq:def-eta},~\eqref{eq:def-omega}, and~\eqref{eq:tested-primal-stepsize} we get
\begin{align*}
Z_{i+1} M_{i+1} 
&=\begin{pmatrix} \phi_i I & - \eta_i V^* \\ - \eta_i V & \psi_{i+1} I \end{pmatrix}.
\end{align*}
For the second claim we observe
\begin{align*}
  M &: = Z_{i+1} M_{i+1} - \begin{pmatrix} \delta \phi_i I & 0 \\ 0 & \psi_{i+1} I - \frac{\eta_i^2}{\phi_i (1-\kappa)} VV^* \end{pmatrix} \\
    &= \begin{pmatrix} (1-\delta) \phi_i I & -\eta_i V^* \\  - \eta_i V &(-\eta_i V ) \left((1-\kappa) \phi_i \right)^{-1} (-\eta_i V)^*  \end{pmatrix}.
\end{align*}
Since $1 > \kappa \geq \delta$ and $(1-\kappa) \phi_i > 0$, we derive the positive semidefiniteness of $M$.
\end{proof}

Hence, we can ensure that $Z_{i+1}M_{i+1}$ is positive semidefinite if we assume
\begin{align}\label{eq:cond-psi}
  \psi_{i+1} &\geq \frac{\eta_i^2}{\phi_i (1-\kappa)} \norm{V}^2.
\end{align}
Note that by $\eta_i = \phi_i \tau_i = \psi_i \sigma_i$ and $\eta_{i+1} = \eta_{i} \omega_{i}^{-1}$ we get that condition~\eqref{eq:cond-psi} enforces
\[\sigma_{i+1} \tau_{i} \leq \frac{(1-\kappa)}{\omega_i} \frac{1}{\norm{V}^2},\]
similar to the widely known stepsize bounds in~\cite{chambolle-pock-first-order}.

Next we investigate the operator $Z_{i+2} M_{i+2} - Z_{i+1} M_{i+1}$  and show that it is easy to evaluate $\norm{u}_{Z_{i+2} M_{i+2} - Z_{i+1} M_{i+1}}^{2}$.

\begin{lemma}[{\cite[Lemma 3.5]{valkonen_acc}}]
Let $i \in \NN$ and assume that conditions~(\ref{eq:adjointpd}),~(\ref{eq:def-eta}),~(\ref{eq:def-omega}) and~(\ref{eq:cond-psi}) are fulfilled and define 
\begin{align}\label{eq:def-Xi}
\Xi_{i+1} := 2 \begin{pmatrix}  \tau_i \mu_G I & \tau_i V^* \\ -\sigma_{i+1} V & \sigma_{i+1}  \mu_{F^*} I  \end{pmatrix}
\end{align}
If the constants $\mu_{G}$ and $\mu_{F^{*}}$ are chosen such that 
\begin{align}
  \phi_{i+1 } &= \phi_i (1 + 2 \tau_i \mu_G), \label{eq:def-phi}\\
  \psi_{i+1} &= \psi_i (1 + 2 \sigma_i \mu_{F^*}),\label{eq:def-psi} 
\end{align}
then it holds that
\begin{align*}
\norm{u}^2_{Z_{i+2}M_{i+2} - Z_{i+1} M_{i+1} } = \norm{u}^2_{Z_{i+1}\Xi_{i+1}}.
\end{align*}
\end{lemma}
\begin{proof}
  We use the expression for $Z_{i+1}M_{i+1}$ from~(\ref{eq:estimate-ZM}) and the conditions~(\ref{eq:adjointpd}),~(\ref{eq:def-eta}),~(\ref{eq:def-phi}),~(\ref{eq:def-psi}) and~(\ref{eq:def-omega}) to get
  \begin{align*}
& Z_{i+1} M_{i+1} + Z_{i+1} \Xi_{i+1}  - Z_{i+2} M_{i+2} \\
=& \begin{pmatrix} \phi_i I & - \eta_i V^* \\ - \eta_i V & \psi_{i+1} I \end{pmatrix} + \begin{pmatrix} \phi_i I & 0 \\ 0 & \psi_{i+1}I \end{pmatrix} \begin{pmatrix} 2 \tau_i \mu_G I & 2\tau_i V^* \\ -2\sigma_{i+1} V & 2 \sigma_{i+1}  \mu_{F^*} I  \end{pmatrix} - 
 \begin{pmatrix} \phi_{i+1} I & - \eta_{i+1} V^* \\ - \eta_{i+1} V & \psi_{i+2} I \end{pmatrix}  \\
 =& \begin{pmatrix} \phi_i I & - \eta_i V^* \\ - \eta_i V & \psi_{i+1} I \end{pmatrix} + \begin{pmatrix} 2 \phi_i \tau_i \mu_G I & 2 \phi_i \tau_i V^* \\ -2\sigma_{i+1} \psi_{i+1} V & 2 \sigma_{i+1} \psi_{i+1} \mu_{F^*} I  \end{pmatrix} - 
 \begin{pmatrix} \phi_{i+1} I & - \eta_{i+1} V^* \\ - \eta_{i+1} V & \psi_{i+2} I \end{pmatrix}  \\
 =&  \begin{pmatrix} \phi_{i+1} I & \eta_{i} V^* \\ (- \eta_i - 2 \eta_{i+1}) V & \psi_{i+2} I \end{pmatrix} -  \begin{pmatrix} \phi_{i+1} I & - \eta_{i+1} V^* \\ - \eta_{i+1} V & \psi_{i+2} I \end{pmatrix} \\
 =&  \begin{pmatrix} 0 & (\eta_i + \eta_{i+1}) V^* \\ - (\eta_i + \eta_{i+1}) V & 0 \end{pmatrix}.
\end{align*}
This shows that $Z_{i+1} M_{i+1} + Z_{i+1} \Xi_{i+1}  - Z_{i+2} M_{i+2}$ is skew-symmetric, and hence, it holds for all $u$ that $\norm{u}^2_{Z_{i+1} ( M_{i+1} + \Xi_{i+1} ) - Z_{i+2} M_{i+2} } = 0$
from which the statement follows.
\end{proof}

\subsection{Technical estimates}

With the preparations from the previous subsection we are in position to estimate the term
\begin{align*}
  D & := \scp{\tilde{H}_{i+1}(u^{i+1})}{u^{i+1} - \hat{u}}_{Z_{i+1}
      W_{i+1}} - \frac12 \norm{u^{i+1} - \hat{u}}_{Z_{i+2}M_{i+2} -
      Z_{i+1}M_{i+1}}^2\\
  & = \scp{\tilde{H}_{i+1}(u^{i+1})}{u^{i+1} - \hat{u}}_{Z_{i+1}
      W_{i+1}} - \frac12 \norm{u^{i+1} - \hat{u}}_{Z_{i+1}\Xi_{i+1}}^2,
\end{align*}
which appears in Theorem~\ref{valkonen_theorem}. Recall that the operator $\tilde H_{i}$ is given
by~(\ref{eq:def-tildeH}), the preconditioner $M_{i}$ is given
by~(\ref{eq:def-precond}), the step length operator $W_{i}$ is given
by~(\ref{eq:def-steplength}), the test operator $Z_{i}$ is defined in~\eqref{eq:def-test-op},
and the operator $\Xi_{i}$ is defined in~\eqref{eq:def-Xi}.
In order to estimate $D$, we assume that both $F^*$ and $G$ are strongly convex functionals and choose the step length appropriately, and state the following estimate for $D$.

\begin{theorem}\label{thm:estimate-D}
  Let $i \in \NN$. Suppose that the
  conditions~(\ref{eq:adjointpd}),~(\ref{eq:def-eta}), (\ref{eq:def-omega}), (\ref{eq:def-phi}), and (\ref{eq:def-psi}) hold,
  that $G, F^*$ are $\gamma_G$/$\gamma_{F^{*}}$-strongly convex,
  respectively, with
  \begin{align}\label{eq:cond-gammaG}
    \gamma_G \geq \frac{\epsilon }{2 \omega_i} \norm{A-V} +
    \mu_G
  \end{align}
  for some $\epsilon>0$.
  Then it holds that
  \begin{align*}
    D \geq \eta_{i+1} ( \gamma_{F^*} - \mu_{F^*} -
    \frac{1+\omega_i}{2 \epsilon} \norm{A-V})  \norm{y^{i+1} -
    \hat{y}}^2 - \frac{\eta_i \epsilon \norm{A-V}}{2} \norm{x^{i+1} -
    x^i}^2.
  \end{align*}
\end{theorem}
\begin{proof}
  We observe
  \begin{align*}
    - \frac12 \norm{u^{i+1} - \hat{u} }^2_{Z_{i+1} \Xi_{i+1}} =&  (\eta_{i+1} - \eta_i) \scp{V (x^{i+1} - \hat{x})}{y^{i+1} - \hat{y} } \\
                                                               &- \eta_i \mu_G \norm{x^{i+1} - \hat{x}}^2 - \eta_{i+1} \mu_{F^*} \norm{y^{i+1} - \hat{y}}^2,
  \end{align*}
  which gives, by definition of $\tilde H_{i+1}$ in~(\ref{eq:def-tildeH}) and $H$ in~(\ref{eq:def_H}),
  \begin{align}\label{eq:def-D}
    \begin{split}
    D =& \scp{\tilde{H}_{i+1}(u^{i+1})}{u^{i+1} - \hat{u}}_{Z_{i+1} W_{i+1}} - \frac12 \norm{u^{i+1} - \hat{u} }^2_{Z_{i+1} \Xi_{i+1}}  \\
    =& \scp{H(u^{i+1})}{u^{i+1} - \hat{u}}_{Z_{i+1} W_{i+1}} \\
        &+ \eta_{i+1} \scp{ (A-V)(x^{i+1} - \bar{x}^{i+1})}{y^{i+1} - \hat{y}} \\
        &+ (\eta_{i+1} - \eta_i) \scp{V (x^{i+1} - \hat{x})}{y^{i+1} - \hat{y} } \\
        &- \eta_i \mu_G \norm{x^{i+1} - \hat{x}}^2 - \eta_{i+1} \mu_{F^*} \norm{y^{i+1} - \hat{y}}^2.
      \end{split}
  \end{align}
  Now we estimate the first term on the right hand side: Since $\hat{u} \in H^{-1}(0)$ with
  $\hat{u} = \begin{pmatrix} \hat{x} \\ \hat{y} \end{pmatrix}$, we
  have
  \begin{align*}
     -V^* \hat{y} \in \partial G(\hat{x}) \qquad\text{and}\qquad
    A \hat{x} \in \partial F^*(\hat{y})
  \end{align*}
  Hence we get
  \begin{align*}
    \scp{H(u^{i+1})}{u^{i+1} - \hat{u}}_{Z_{i+1} W_{i+1}} =&\; \scp{H(u^{i+1})- H(\hat{u})}{u^{i+1} - \hat{u}}_{Z_{i+1} W_{i+1}} \\
    =&\; \eta_i \scp{\partial G(x^{i+1}) - \partial G(\hat x)}{x^{i+1}-\hat{x}} \\
    &\quad+ \eta_{i+1} \scp{ \partial F^*(y^{i+1}) - \partial F^{*}(\hat y)}{y^{i+1}-\hat{y}} \\
                                                           &\quad+ \eta_i \scp{V^*(y^{i+1} - \hat{y})}{x^{i+1} - \hat{x}} \\ &\quad+ \eta_{i+1} \scp{A(\hat{x} - x^{i+1})}{y^{i+1}-\hat{y}}.
  \end{align*}
  Now the strong convexity of $G$ and $F^*$ with constants $\gamma_{G}$ and $\gamma_{F^{*}}$, respectively, results in the inequality
  \begin{align*}
    \scp{H(u^{i+1})}{u^{i+1} - \hat{u}}_{Z_{i+1} W_{i+1}} \geq&\; \eta_i \gamma_G \norm{x^{i+1} - \hat{x}}^2 + \eta_{i+1} \gamma_{F^*} \norm{y^{i+1} - \hat{y}}^2 \\
                                                              &\quad+ \eta_i \scp{ V(x^{i+1} - \hat{x}) }{ y^{i+1} - \hat{y} } \\ &\quad+ \eta_{i+1} \scp{ A(\hat{x} - x^{i+1}) }{ y^{i+1} - \hat{y} }.
  \end{align*}
  Plugging this into the definition of $D$ in~(\ref{eq:def-D}) and collecting terms gives
  \begin{align*}
    D \geq&\; \eta_i (\gamma_G-\mu_G) \norm{x^{i+1} - \hat{x}}^2 + \eta_{i+1} (\gamma_{F^*} - \mu_{F^*}) \norm{y^{i+1} - \hat{y}}^2 \\
          &\quad+ \eta_i \scp{ V(x^{i+1} - \hat{x}) }{ y^{i+1} - \hat{y} } + \eta_{i+1} \scp{ A(\hat{x} - x^{i+1}) }{ y^{i+1} - \hat{y} } \\
          &\quad+ \eta_{i+1} \scp{ V(\bar{x}^{i+1} - x^{i+1}) }{ y^{i+1} - \hat{y} } + (\eta_{i+1} - \eta_i) \scp{ V(x^{i+1} - \hat{x} ) }{ y^{i+1} - \hat{y} } \\
    =&\;  \eta_i (\gamma_G-\mu_G) \norm{x^{i+1} - \hat{x}}^2 + \eta_{i+1} (\gamma_{F^*} - \mu_{F^*}) \norm{y^{i+1} - \hat{y}}^2 \\
          &\quad+ \eta_{i+1} \scp{ (A-V)(\hat{x} - x^{i+1}) }{ y^{i+1} - \hat{y} } + \eta_{i+1} \scp{ (A-V)( x^{i+1} - \bar{x}^{i+1}) }{ y^{i+1} - \hat{y} } .
  \end{align*}
  We use the extrapolation $\bar{x}^{i+1} - x^{i+1} = \omega_i (x^{i+1} - x^i)$ to get for every $\epsilon>0$
  \begin{align*}
    \scp{ (A-V)(x^{i+1} - \bar{x}^{i+1}) }{ y^{i+1} - \hat{y} } &= \omega_i \scp{ (A-V)(x^i - x^{i+1}) }{ y^{i+1} - \hat{y} } \\
                                                                &\geq - \omega_i \norm{A-V} \norm{x^i -x^{i+1} } \norm{ y^{i+1} - \hat{y} } \\
                                                                &\geq - \frac{\omega_i \norm{A-V} }{2} \left( \epsilon \norm{x^i -x^{i+1} }^2 + \frac{1}{\epsilon}  \norm{ y^{i+1} - \hat{y} }^2 \right)
  \end{align*}
 by Young's inequality. Similarly we derive
  \begin{align*}
    \scp{(A-V)(\hat{x} - x^{i+1}) }{ y^{i+1} - \hat{y} } &\geq - \norm{A-V} \norm{\hat{x} - x^{i+1} } \norm{ y^{i+1} - \hat{y} } \\
                                                         &\geq - \frac{\norm{A-V}}{2} \left( \epsilon \norm{\hat{x} - x^{i+1} }^2 +  \frac{1}{\epsilon} \norm{ y^{i+1} - \hat{y} }^2 \right).
  \end{align*}
  With these two estimates we continue to lower bound $D$ and arrive at
  \begin{align*}
    D \geq & \eta_i (\gamma_G-\mu_G) \norm{x^{i+1} - \hat{x}}^2 + \eta_{i+1} \left(\gamma_{F^*} - \mu_{F^*}\right) \norm{y^{i+1} - \hat{y}}^2  \\
           & - \eta_{i+1} \epsilon \frac{\norm{A-V}}{2}  \norm{ x^{i+1} - \hat{x} }^2  - \eta_{i+1} \frac{\norm{A-V}}{2 \epsilon}   \norm{ y^{i+1} - \hat{y} }^2  \\
           & - \eta_{i+1} \epsilon \frac{\omega_i \norm{A-V} }{2}  \norm{x^i -x^{i+1} }^2 - \eta_{i+1} \frac{\omega_i \norm{A-V} }{2 \epsilon}   \norm{ y^{i+1} - \hat{y} }^2  \\
    =& \eta_{i+1} (\gamma_{F^*} - \mu_{F^*} - \frac{(1+\omega_i) \norm{A-V}}{2 \epsilon} ) \norm{y^{i+1} - \hat{y}}^2 \\
           & +  [ \eta_i (\gamma_G-\mu_G) - \eta_{i+1} \epsilon \frac{\norm{A-V}}{2}] \norm{x^{i+1} - \hat{x}}^2  \\
           & - \eta_{i+1} \epsilon \frac{\omega_i \norm{A-V}}{2} \norm{x^{i+1} - x^i }^2 \\
    \geq& \left[\eta_{i+1} ( \gamma_{F^*} - \mu_{F^*} - \frac{1+\omega_i}{2 \epsilon} \norm{A-V}) \right] \norm{y^{i+1} - \hat{y}}^2 - \frac{\eta_i \epsilon \norm{A-V}}{2} \norm{x^{i+1} - x^i}^2.
  \end{align*}
\end{proof}

The next lemma provides an estimate for $-\Delta_{i+1}$.

\begin{lemma}\label{lem:estimate_D}
  Let $i \in \NN$, the assumptions of Theorem~\ref{thm:estimate-D} be fulfilled and assume furthermore that~(\ref{eq:cond-psi}), (\ref{eq:cond-gammaG}) and
  \begin{align}
    \phi_i &\geq \frac{\epsilon \eta_i}{\delta} \norm{A-V} \label{eq:cond-phi} 
  \end{align} hold with $0<\delta\leq\kappa <1$.
  Define
  \begin{align}\label{eq:def-S}
  S_{i+1} = \begin{pmatrix} (\delta \phi_i - \epsilon \eta_i
    \norm{A-V}) I & 0 \\ 0 & \psi_{i+1} I - \frac{\eta_i^2}{\phi_i
      (1-\kappa)} VV^*\end{pmatrix},
  \end{align}
  and assume that
  \begin{align*}
    \frac12 \norm{u^{i+1} - u^i }_{S_{i+1}}^2 + \eta_{i+1}
    (\gamma_{F^*} - \mu_{F^*} - \frac{1+\omega_i}{2 \epsilon} \norm{A-V}
    ) \norm{ y^{i+1} - \hat{y} }^2 \geq - \Delta_{i+1},
  \end{align*}
  is fulfilled. Then it holds that
  \begin{align*}
    -\Delta_{i+1} \leq \frac12 \norm{ u^{i+1} - u^i }_{Z_{i+1} M_{i+1}}^2 + D.
  \end{align*}
\end{lemma}
\begin{proof}
In the first step, we rewrite
\begin{align*}
S_{i+1} &= \begin{pmatrix} (\delta \phi_i - \epsilon \eta_i \norm{A-V}) I & 0 \\ 0 & \psi_{i+1} I - \frac{\eta_i^2}{\phi_i (1-\kappa)} VV^*\end{pmatrix} \\
&= \underbrace{\begin{pmatrix} \delta \phi_i I & 0 \\ 0 & \psi_{i+1} I - \frac{\eta_i^2}{\phi_i (1-\kappa)} VV^*\end{pmatrix}}_{\leq Z_{i+1}M_{i+1}} - \begin{pmatrix} \epsilon \eta_i \norm{A-V} I & 0 \\ 0 & 0 \end{pmatrix},
\end{align*}
so Lemma 2.2 gives
\begin{align*}
  \frac12 \norm{ u^{i+1} - u^i }_{S_{i+1}}^2 \leq \frac12  \norm{ u^{i+1} - u^i }_{Z_{i+1} M_{i+1}}^2 - \frac{\eta_i \epsilon \norm{A-V}}{2} \norm{ x^{i+1} - x^i }^2.
\end{align*}
Hence, 
\begin{align*}
-\Delta_{i+1} &\leq \frac12 \norm{u^{i+1} - u^i }_{S_{i+1}}^2 + \eta_{i+1} (\gamma_{F^*} - \mu_{F^*} - \frac{1+\omega_i}{2 \epsilon} \norm{A-V} ) \norm{ y^{i+1} - \hat{y} }^2 \\
              &\leq \frac12  \norm{ u^{i+1} - u^i }_{Z_{i+1} M_{i+1}}^2
   - \frac{\eta_i \epsilon \norm{A-V}}{2} \norm{ x^{i+1} - x^i }^2 \\ &\qquad+ \eta_{i+1} (\gamma_{F^*} - \mu_{F^*} - \frac{1+\omega_i}{2 \epsilon} \norm{A-V} ) \norm{ y^{i+1} - \hat{y} }^2
\end{align*}
which shows the claim since Theorem~\ref{thm:estimate-D} shows that the last two terms are a lower bound for $D$.
\end{proof}

Now we state the following abstract convergence result which concludes this section:

\begin{theorem}\label{thm:estimate-norm-uN-hatu}
  Suppose that the step length
  conditions~(\ref{eq:adjointpd}),~(\ref{eq:def-eta}), (\ref{eq:def-omega}), (\ref{eq:cond-psi})--(\ref{eq:def-psi}) and (\ref{eq:cond-phi}) hold with
  $\epsilon >0$, $0 < \delta \leq \kappa < 1$ and for all $i \in \NN$. Additionally suppose
  that $G, F^*$ are $\gamma_G$/$\gamma_{F^{*}}$-strongly convex,
  respectively, and that (\ref{eq:cond-gammaG}) holds.
  Furthermore, let $S_{i+1}$ be defined by~(\ref{eq:def-S}) and $\hat u =
  \begin{pmatrix}
    \hat x\\\hat y
  \end{pmatrix}$ fulfill $0\in\tilde H(\hat u)$.
  Then it holds
  \begin{align*}
    \frac12 \norm{u^N - \hat{u}}^2_{Z_{N+1}M_{N+1}} \leq \frac12 \|u^0-\hat{u}\|^2_{Z_1M_1} + \sum_{i=0}^{N-1} \Delta_{i+1}
  \end{align*}
  for every $\Delta_{i}$ which fulfills
  \begin{align*}
  \frac12 \norm{u^{i+1} - u^i }_{S_{i+1}}^2 + \eta_{i+1}
  (\gamma_{F^*} - \mu_{F^*} - \frac{1+\omega_i}{2 \epsilon} \norm{A-V}
  ) \norm{ y^{i+1} - \hat{y} }^2 \geq - \Delta_{i+1}.
  \end{align*}

\end{theorem}
\begin{proof}
We recognize from Lemma \ref{lem:estimate_D} that
\begin{align*}
    - \Delta_{i+1} 
                   \leq& \frac12 \norm{ u^{i+1} - u^i }_{Z_{i+1} M_{i+1}}^2 + \scp{ \tilde{H}_{i+1}(u^{i+1}) }{ u^{i+1} - \hat{u} }_{Z_{i+1} W_{i+1} } \\ &- \frac12 \norm{ u^{i+1} - \hat{u} }_{Z_{i+1} \Xi_{i+1}}^2.
\end{align*}
Using the equality
  \begin{align*}
  \norm{u^{i+1} - \hat{u} }_{Z_{i+1} \Xi_{i+1}}^2 = \norm{u^{i+1} - \hat{u} }_{Z_{i+2} M_{i+2} - Z_{i+1} M_{i+1}}^2
  \end{align*}
  we get
\begin{align*}
    - \Delta_{i+1} \leq& \frac12 \norm{ u^{i+1} - u^i }_{Z_{i+1} M_{i+1}}^2 + \scp{ \tilde{H}_{i+1}(u^{i+1}) }{ u^{i+1} - \hat{u} }_{Z_{i+1} W_{i+1} }  \\ & - \frac12 \norm{ u^{i+1} - \hat{u} }_{Z_{i+2} M_{i+2} - Z_{i+1} M_{i+1}}^2.
  \end{align*} 
  Rearranging these terms leads to
  \begin{align*}
    \scp{ \tilde{H}_{i+1}(u^{i+1}) }{ u^{i+1} - \hat{u} }_{Z_{i+1} W_{i+1} } &\geq \frac12 \norm{ u^{i+1} - \hat{u} }_{Z_{i+2} M_{i+2} - Z_{i+1} M_{i+1}}^2 \\ & \quad\quad-  \frac12 \norm{ u^{i+1} - u^i }_{Z_{i+1} M_{i+1}}^2 - \Delta_{i+1}
  \end{align*}
 and thus all conditions of Theorem \ref{valkonen_theorem} are satisfied, and the result follows.
\end{proof}

\section{Convergence rates}
\label{sec:convergence}
With the results from the previous section, we are in position to prove convergence of~(\ref{eq:cpmm}) under easily verifiable conditions.
We assume that both $G$ as well as $F^*$ are proper, strongly convex, and lower-semicontinuous functions, and that the stepsizes can be chosen such that we obtain linear convergence. 
We proceed as follows: First we derive conditions under which we can guarantee linear convergence, and then we show how to select the parameters in order to obtain a choice of the stepsizes that is simply to apply in practice. 

\begin{theorem}\label{thm:linear_conv_rate}
  Choose  $\mu_G > 0, \mu_{F^*} > 0$ and suppose that $G$ is $\gamma_G$-strongly convex and $F^*$ is $\gamma_{F^*}$-strongly convex with
  \begin{align}\label{eq:cond-tauG-tauF*}
    \gamma_G \geq \frac{\epsilon}{2 \omega_i} \norm{A-V} + \mu_G, \;\; \gamma_{F^*} \geq \frac{1+\omega_i}{2 \epsilon} \|A-V\| + \mu_{F^*}
  \end{align}
for some $\epsilon > 0$. Furthermore let $u^{N} = (x^{N},y^{N})^{T}$ be generated by the Chambolle-Pock method with mismatched adjoint (\ref{eq:cpmm}) and $\hat u = (\hat x,\hat y)^{T}$ be a fixed point of this iteration.
Then  with constant step lengths
\begin{align*}
\tau_i = \tau &:= \min \left\{ \frac{\epsilon^{-1} \delta}{\norm{A-V}}, \sqrt{ \frac{(1-\kappa)\mu_{F^*}}{\norm{V}^2 \mu_G}} \right\}, \\
\sigma_i = \sigma &:= \frac{\mu_{G} }{\mu_{F^*} } \tau, \\
\omega_i = \omega &:= (1+2 \tau \mu_G)^{-1}
\end{align*}
for some $0 \leq \delta \leq \kappa < 1$ it holds that $\norm{u^N - \hat{u}}^2 =\mathcal{O}(\omega^N)$.
\end{theorem}
\begin{proof}
We set $\phi_0 := \frac{1}{\tau}$ and $\psi_0 := \frac{1}{\sigma}$ and by~(\ref{eq:def-phi}) we get
\[
  \phi_i \tau = \underbrace{\phi_0 \tau}_{=1} (1+2\tau \mu_G)^i = (1+2\tau \mu_G)^i
\]
By~(\ref{eq:adjointpd}),~(\ref{eq:def-eta}) and~(\ref{eq:def-psi}) we get
\begin{align*} 
\psi_{i+1} \sigma &= \psi_ {i} \sigma (1+2\sigma \mu_{F^*})= \psi_ {i} \sigma (1+2\tau \mu_{G}) \\
&= \underbrace{\psi_0 \sigma}_{=1} (1+2\tau \mu_{G})^{i+1} = (1+2\tau \mu_{G})^{i+1}.
\end{align*}
Hence, again by~(\ref{eq:def-eta}), we get
\[
  \eta_i = (1+2\tau \mu_{G})^i
\]
and, by~(\ref{eq:def-omega})
\[
  \omega = \frac{\eta_i}{\eta_{i+1}} = \frac{1}{1+2\tau \mu_{G}}.
\]
Now we claim that the matrix $S_{i+1}$ defined in~(\ref{eq:def-S}) is
positive semidefinite, which is equivalent to the conditions
$\delta\phi_{i} - \eta_{i}\epsilon\norm{A-V}\geq 0$ and
$\psi_{i+1}\geq \eta_{i}^{2}\norm{V}^{2}/(\phi_{i}(1-\kappa))$. The
first condition is fulfilled since $\tau\leq \delta/(\epsilon\norm{A-V})$ holds by the assumption in the theorem.
The second condition follows from the assumption $\tau^2 \leq (1-\kappa)\mu_{F^{*}}/(\norm{V}^{2}\mu_{G})$, since
\begin{align*}
\tau^2 \leq \frac{(1-\kappa) \mu_{F^*}}{\norm{V}^2 \mu_G} & \overset{\sigma = \frac{\mu_G}{\mu_{F^*}} \tau}{\Leftrightarrow} \tau \sigma \leq \frac{1-\kappa}{\norm{V}^2}  \\
& \overset{\eta_i = \sigma \psi_{i} = \tau \phi_i}{\Leftrightarrow}  \frac{\eta_{i+1} \eta_i}{\psi_{i+1} \phi_i} \leq \frac{1-\kappa}{\norm{V}^2} \\
& \overset{\omega = \eta_i / \eta_{i+1}}{\Leftrightarrow}  \frac{\eta_{i}^2}{\omega \psi_{i+1} \phi_i} \leq \frac{1-\kappa}{\norm{V}^2} 
\Leftrightarrow \psi_{i+1} \geq \frac{\eta_i^2 \norm{V}^2}{\phi_i (1- \kappa) \omega} \geq \frac{\eta_i^2 \norm{V}^2}{\phi_i (1- \kappa) }.
\end{align*}
Consequently, $S_{i+1}$ is positive semidefinite, so we can choose
$\Delta_{i+1} = 0$ for $i \in \NN$.
As a result, Theorem~\ref{thm:estimate-norm-uN-hatu} and Lemma~\ref{lem:estimate_D} results in
\begin{align*}
  \norm{ u^0 - \hat{u}}_{Z_1M_1}^2 &\geq \| u^N - \hat{u} \|^2_{Z_{N+1} M_{N+1}} \\
                                  &\geq \delta \phi_N \| x^N - \hat{x} \|^2 + \underbrace{\left( \psi_{N+1} - \frac{\eta_N^2 \|V\|^2}{\phi_N (1-\kappa)} \right)}_{\geq 0} \|y^N - \hat{y}\|^2 \\
                                  & = \phi_{N}\Big(\delta\norm{x^{N}-\hat x}^{2} + \big(\tfrac{\psi_{N+1}}{\phi_{N}} - \tfrac{\eta_{N}^{2}\norm{V}^2}{\phi_{N}^{2}(1-\kappa)}\big)\norm{y^{N}-\hat y}^{2}\Big).
\end{align*}
Using the properties 
$\eta_N = \tau \phi_N = \sigma \psi_N$ and $\eta_{N+1} = \omega \eta_N$ we arrive at
\begin{align*}
  \norm{u^0 - \hat{u}}_{Z_1M_1}^2 & \geq  \phi_{N}\Big(\delta\norm{x^{N}-\hat x}^{2} + \big(\tfrac{\psi_{N+1}}{\phi_{N}} - \tfrac{\eta_{N}^{2}\norm{V}^2}{\phi_{N}^{2}(1-\kappa)}\big)\norm{y^{N}-\hat y}^{2}\Big) \\
                                  & = (1+2\tau\mu_{G})^{N}\left(\tfrac{\delta}{\tau}\norm{x^{N}-\hat x}^{2} + \big(\tfrac{\tau}{\sigma}(1+2\tau\mu_{G}) - \frac{ \tau^2 \norm{V}^2}{(1-\kappa)} \big)\norm{y^{N}-\hat y}^{2}\right)
\end{align*}
which proves the claim.
\end{proof}

This above theorem is not immediately practical, since it is unclear, if all parameters can be chosen such that all conditions are fulfilled. Hence, we now derive a method that allows for a feasible  choice of the parameters.

 If we plug the definition of $\omega$ into conditions~(\ref{eq:cond-tauG-tauF*}), we get
  \begin{align}
    \gamma_{G}& \geq \tfrac{\epsilon}{2}(1+2\tau\mu_{G})\norm{A-V} + \mu_{G}\label{eq:muG-var}\\
    \gamma_{F^{*}}& \geq \tfrac{1+\tau\mu_{G}}{\epsilon(1+2\tau\mu_{G})}\norm{A-V} + \mu_{F^{*}}.\label{eq:muFs-var}
  \end{align}
  Since  $\tfrac12\leq \tfrac{1+t}{1+2t}\leq 1$ for $t>0$ , ~(\ref{eq:muFs-var}) is fulfilled, if
  \begin{align}\label{eq:muFs-var2}
    \gamma_{F^{*}} = \tfrac{\norm{A-V}}{\epsilon} + \mu_{F^{*}}.
  \end{align}
  We express the quantities $\mu_{G}$ and $\mu_{F^{*}}$ by $\mu_{F^{*}} = a\gamma_{F^{*}},\ \mu_{G} = b\gamma_{G}$ with $0\leq a,b\leq 1$.
  Then, we can use~(\ref{eq:muFs-var2}) to get a valid value for $\epsilon$, namely
  \begin{align*}
    \epsilon = \tfrac{\norm{A-V}}{(1-a)\gamma_{F^{*}}}.
  \end{align*}
  Furthermore, we observe that it is always beneficial to
  choose $\delta$ as large as possible, i.e. we set $\delta=\kappa$. 
  Additionally, using this $\epsilon$ and $\mu_{G}$ in~(\ref{eq:muG-var}), we get
  \begin{align*}
    \gamma_{G}& \geq \tfrac{1+2\tau b\gamma_{G}}{2(1-a)\gamma_{F^{*}}}\norm{A-V}^{2} + b\gamma_{G},
  \end{align*}
  which gives the inequality
  \begin{align}\label{eq:tau-upper-est}
    \tau \leq \tfrac{(1-b)(1-a)\gamma_{F^{*}}}{\norm{A-V}^{2}b} - \tfrac{1}{2b\gamma_{G}}.
  \end{align}
  On the other hand, we plug in $\delta=\kappa$, and the values for $\epsilon$, $\mu_{F^{*}}$ and $\mu_{G}$ into the definition of $\tau$ in Theorem~\ref{thm:linear_conv_rate} and get
  \begin{align*}
    \tau = \min\left\{\frac{\kappa(1-a)\gamma_{F^{*}}}{\norm{A-V}^{2}}, \sqrt{ \frac{(1-\kappa)a\gamma_{F^*}}{\norm{V}^2 b\gamma_G}} \right\}.
  \end{align*}
  For all choices of $\kappa$ and $a$, there exists a small $b \in (0,1)$, such that the minimal value is attained at the left expression in minimum.
  Hence, by choice, we choose our parameters in such a way, that 
  \begin{align*}
    \tau =\frac{\kappa(1-a)\gamma_{F^{*}}}{\norm{A-V}^{2}},
  \end{align*}
  so we have that
  \begin{align}\label{eq:defTauInequal}
    \tau^2 = \tfrac{\kappa^2 (1-a)^2 \gamma_{F^*}^2}{\norm{A-V}^4} \leq \tfrac{(1-\kappa) a\gamma_{F^{*}}}{\norm{V}^2 b \gamma_G}
  \end{align}
  must hold in the definition.
  Hence, we have to find $a, b$ and $\kappa$, such that
  \begin{align*}
     \tfrac{\kappa (1-a) \gamma_{F^*}}{\norm{A-V}^2} = \tau \overset{(\ref{eq:tau-upper-est})} \leq \tfrac{(1-b)(1-a)\gamma_{F^{*}}}{\norm{A-V}^{2}b} - \tfrac{1}{2b\gamma_{G}},
  \end{align*}
  which is equivalent to
   \begin{align}\label{eq:kappa-upper-bound}
    \kappa \leq \tfrac{1-b}{b} - \tfrac{\norm{A-V}^2}{2 b (1-a) \gamma_{F^*} \gamma_G} = \tfrac{1}{b} \left( 1-b - \tfrac{\norm{A-V}^2}{2 (1-a) \gamma_{F^*} \gamma_G}  \right).
  \end{align}
 Clearly, the upper bound increases with decreasing the value of $b$. Hence, we restrict ourselves to $b \leq \tfrac{1}{2}$ and use our degrees of freedom to set $a = \tfrac{1}{2}$. Now~(\ref{eq:defTauInequal}) turns into
 \begin{align}
    \tfrac{\kappa^2 \gamma_{F^*}}{4 \norm{A-V}^4} \leq \tfrac{1}{b} \tfrac{1-\kappa}{2 \norm{V}^2 \gamma_G},
  \end{align}
  which is equivalent to the condition
  \begin{align}
   b \leq \tfrac{1- \kappa}{\kappa^2} \tfrac{\norm{A-V}^4}{\norm{V}^2} \tfrac{2}{\gamma_{F^*} \gamma_G}
  \end{align}
  on $b$. To satisfy~(\ref{eq:kappa-upper-bound}), we require $b \leq \frac12$ and
  \begin{align}\label{eq:defKappaPos}
	\kappa \leq \tfrac{1}{b} \left(\tfrac{1}{2} - \tfrac{\norm{A-V}^2}{\gamma_{F^*} \gamma_G} \right).
  \end{align}
  The later is positive, whenever
  \begin{align}
  2 \norm{A-V}^2 < \gamma_{F^*} \gamma_G
  \end{align}
  is fulfilled. So by (\ref{eq:defKappaPos}) we are able to find a small enough $b$ for every $\kappa \in (0,1)$, such that the corresponding inequality holds. In conclusion, we have proven:
 \begin{theorem}\label{thm:existence_linear_convergence}
   Suppose that $G, F^*$ are $\gamma_G/\gamma_{F^*}$-strongly convex functions fulfilling
   $$\gamma_{F^*} \gamma_G > 2 \norm{A-V}^2.$$ 
   Define
   $$b = \min\left\{\tfrac{1}{2}, \tfrac{1}{\kappa} \left(\tfrac{1}{2} - \tfrac{\norm{A-V}^2}{\gamma_{F^*} \gamma_G} \right), \tfrac{1-\kappa}{\kappa^2} \tfrac{\norm{A-V}^4}{\norm{V}^2} \tfrac{2}{\gamma_{F^*} \gamma_G}\right\}$$
   for some arbitrary $\kappa \in (0,1)$. Furthermore let $\hat{u} = (\hat x,\hat y)^{T}$ be a fixed point of the Chambolle-Pock method with mismatched adjoint and constant step lengths
\begin{align*}
\tau_i = \tau &:= \sqrt{ \frac{(1-\kappa)\gamma_{F^*}}{2 b \norm{V}^2 \gamma_G}}, \\
\sigma_i = \sigma &:= 2 b \frac{\gamma_{G} }{\gamma_{F^*} } \tau, \\
\omega_i = \omega &:= (1+2 b \tau \gamma_G)^{-1}.
\end{align*}
Then the iterates $(u^{N})_{N \in \NN}$ converge
with $\norm{u^N - \hat{u}}^2$ decaying to zero at the rate $\mathcal{O}(\omega^N)$.
 \end{theorem}
 As a consequence of the freedom of choice for the value of $\kappa \in (0,1)$, we can choose $\kappa \leq \frac12$ small enough such that $b$ in Theorem~\ref{thm:existence_linear_convergence} equals $\frac12$. This leads to the following corollary.
 \begin{corollary}\label{cor:simple-parameter-choice}
 Suppose that $G, F^*$ are $\gamma_G/\gamma_{F^*}$-strongly convex functions fulfilling
 \[
   \gamma_{F^*} \gamma_G > 2 \norm{A-V}^2.
 \]
   Let
   \[
    0< \kappa \leq \min \left\{\tfrac{1}{2}, 1- \tfrac{2 \norm{A-V}^2}{\gamma_{F^*} \gamma_G}, \tfrac{\norm{A-V}^2}{\norm{V}} \sqrt{\tfrac{2}{\gamma_{F^*} \gamma_G}}\right\}.
   \]
  Furthermore let $\hat{u} = (\hat x,\hat y)^{T}$ be a fixed point of the Chambolle-Pock method with mismatched adjoint and constant step lengths
\begin{align*}
\tau_i = \tau &:= \sqrt{ \frac{(1-\kappa)\gamma_{F^*}}{\norm{V}^2 \gamma_G}}, \\
\sigma_i = \sigma &:= \frac{\gamma_{G} }{\gamma_{F^*} } \tau, \\
\omega_i = \omega &:= (1+\tau \gamma_G)^{-1}.
\end{align*}
Then the iterates converge
with $\norm{u^N - \hat{u}}^2$ decaying to zero at the rate $\mathcal{O}(\omega^N)$.
\end{corollary}

\begin{proof}
By choosing $\kappa$ as stated, one shows by a routine calculation that one gets $b=1/2$ in Theorem~\ref{thm:existence_linear_convergence}.
\end{proof}

As a consequence of Corollary~\ref{cor:simple-parameter-choice} we get a simple parameter choice method. For fast convergence one needs small $\omega$, i.e. one wants a large $\tau$. Hence, smaller $\kappa$ is better and thus, one chooses $\kappa$ positive but small (e.g. $\kappa=10^{-5}$). Note that $\kappa=0$ is not covered by our theory.
 
 Since many problems involve only one strongly convex function while the other function only remains to be convex we investigate if we can prove convergence of the Chambolle-Pock method with mismatch with the approach in this paper. To that end, we start again at inequality  (\ref{eqn:secondcond}) and investigate under which conditions we have non-positivity of $\Delta_{i+1}$. Using the definition of $\tilde H$ from~\eqref{eq:def-tildeH}, $H$ from \eqref{eq:def_H}, the definition of $Z_{i+1}$ from~\eqref{eq:def-test-op}, the one of $W_{i+1}$ from~\eqref{eq:def-steplength} and the relations for $\eta,\tau,\sigma,\xi,\phi,\psi$ from Section~\ref{sec:conv} we get from the monotonicity of the subgradient
\begin{align*} \allowdisplaybreaks
	 &\scp{\tilde{H}_{i+1}(u^{i+1})}{u^{i+1} - \hat{u}}_{Z_{i+1} W_{i+1}}  - \frac12 \norm{u^{i+1} - \hat{u}}_{Z_{i+2}M_{i+2} - Z_{i+1}M_{i+1}}^2 \\
	&\quad+ \frac12 \norm{u^{i+1}-u^i}^2_{Z_{i+1}M_{i+1}} \\
	&= \scp{H(u^{i+1}) - H(\hat{u})}{u^{i+1} - \hat{u}}_{Z_{i+1} W_{i+1}} \\
        &\quad+ \eta_{i+1} \scp{ (A-V)(x^{i+1} - \bar{x}^{i+1})}{y^{i+1} - \hat{y}} \\
        &\quad+ (\eta_{i+1} - \eta_i) \scp{V (x^{i+1} - \hat{x})}{y^{i+1} - \hat{y} } \\
        &\quad- \eta_i \mu_G \norm{x^{i+1} - \hat{x}}^2 - \eta_{i+1} \mu_{F^*} \norm{y^{i+1} - \hat{y}}^2 + \frac12 \norm{u^{i+1}-u^i}^2_{Z_{i+1}M_{i+1}} \\
        &\geq \eta_i \gamma_G \norm{x^{i+1} - \hat{x}}^2 + \eta_{i+1} \gamma_{F^*} \norm{y^{i+1} - \hat{y}}^2 \\
        &\quad+ \eta_i \scp{V^*(y^{i+1} - \hat{y})}{x^{i+1}-\hat{x}} - \eta_{i+1} \scp{y^{i+1} - \hat{y}}{A(x^{i+1} - \hat{x})} \\
        &\quad- \eta_i \scp{y^{i+1} - \hat{y}}{(A-V)(x^{i+1} - x^{i})} + (\eta_{i+1} - \eta_{i}) \scp{V^* (y^{i+1} - \hat{y})}{x^{i+1} - \hat{x}} \\
        &\quad- \eta_i \mu_G \norm{x^{i+1} - \hat{x}}^2 - \eta_{i+1} \mu_{F^*} \norm{y^{i+1} - \hat{y}}^2 \\
        &\quad+\phi_{i} \norm{x^{i+1} - x^{i}}^2 + \psi_{i+1} \norm{y^{i+1} - y^{i}}^2 - 2 \eta_{i} \scp{V^{*} (y^{i+1} - y^{i})}{x^{i+1}-x^{i}} \\
        &= \eta_{i} (\gamma_G - \mu_G) \norm{x^{i+1} - \hat{x}}^2 + \eta_{i+1} (\gamma_{F^*} - \mu_{F^*}) \norm{y^{i+1} - \hat{y}}^2 \\
        &\quad- \eta_{i+1} \scp{y^{i+1} - \hat{y}}{(A-V)(x^{i+1} - \hat{x})} - \eta_i \scp{y^{i+1} - \hat{y}}{(A-V)(x^{i+1} - x^{i})} \\
        &\quad+\phi_{i} \norm{x^{i+1} - x^{i}}^2 + \psi_{i+1} \norm{y^{i+1} - y^{i}}^2 - 2 \eta_{i} \scp{V^{*} (y^{i+1} - y^{i})}{x^{i+1}-x^{i}} \\
        &= \eta_{i} (\gamma_G - \mu_G) \norm{x^{i+1} - \hat{x}}^2 + \eta_{i+1} (\gamma_{F^*} - \mu_{F^*}) \norm{y^{i+1} - \hat{y}}^2 \\
        &\quad- \eta_{i+1} \scp{y^{i+1} - \hat{y}}{(A-V)[x^{i+1} - \hat{x} + \omega_i (x^{i+1} - x^{i})]} \\
        &\quad+\phi_{i} \norm{x^{i+1} - x^{i}}^2 + \psi_{i+1} \norm{y^{i+1} - y^{i}}^2 - 2 \eta_{i} \scp{V^{*} (y^{i+1} - y^{i})}{x^{i+1}-x^{i}}
 \end{align*}
Using the abbreviations
\begin{align*}
\begin{array}{llllll}
a = \norm{x^{i+1} - \hat{x}}, &  b = \norm{y^{i+1} - \hat{y}}, & c = \norm{x^{i+1} - x^{i}} & \text{and} & d = \norm{y^{i+1} - y^{i}}
\end{array}
\end{align*}
as well as the Cauchy-Schwarz inequality, we can finally bound 
\begin{align*}
  &\scp{\tilde{H}_{i+1}(u^{i+1})}{u^{i+1} - \hat{u}}_{Z_{i+1} W_{i+1}}  - \frac12 \norm{u^{i+1} - \hat{u}}_{Z_{i+2}M_{i+2} - Z_{i+1}M_{i+1}}^2 \\
	&\quad+ \frac12 \norm{u^{i+1}-u^i}^2_{Z_{i+1}M_{i+1}} \\
        &\geq  \eta_{i} (\gamma_G - \mu_G) a^2 + \eta_{i+1} (\gamma_{F^*} - \mu_{F^*}) b^2 - \eta_{i+1} \norm{A-V} ab - \eta_i \norm{A-V} bc \\
        &\quad+\phi_{i} c^2 + \psi_{i+1} d^2 - 2 \eta_{i} \norm{V} cd.
\end{align*}
This is a quadratic polynomial in the four variables $a,b,c,d$ and hence, non-negativity of this expression is implied by positiv semidefiniteness of the quadratic form $(a,b,c,d) Q (a,b,c,d)^T$ with
 \[Q = \left( \begin{array}{cccc}
 \eta_{i} (\gamma_G - \mu_G) & - \frac12 \eta_{i+1} \norm{A-V} & 0 & 0 \\
 - \frac12 \eta_{i+1} \norm{A-V} & \eta_{i+1} (\gamma_{F^*} - \mu_{F^*}) & - \frac12 \eta_i \norm{A-V} & 0 \\
 0 &  - \frac12 \eta_i \norm{A-V} & \phi_i & - \eta_i \norm{V} \\
 0 & 0 & - \eta_{i} \norm{V} & \psi_{i+1}
\end{array} \right) \]          
However, the conditions for positive semidefiniteness of $Q$ involve the inequality
\[
  \omega_{i}  (\gamma_G - \mu_G)  (\gamma_{F^*} - \mu_{F^*}) - \frac14 \norm{A-V}^2 \geq 0
\]
and if we have $\norm{A-V}\neq0$, i.e.~there is mismatch, then this implies that $\gamma_{G},\gamma_{F^{*}}>0$ is necessary for $\mu_{G},\mu_{F^{*}}$ to exist. Hence, we can not prove convergence of the Chambolle-Pock method with mismatch with the techniques of this paper if $G$ or $F^{*}$ is not strongly convex.

Here is a counterexample, that the Chambolle-Pock method with mismatch may actually diverge if there is mismatch and one of the functions is strongly convex while the other is not.
\begin{example}
Let $A =
\begin{pmatrix}
  1 & 1
\end{pmatrix}\in\RR^{1\times 2}$, $F(y) = (y-z)^{2}/2$ for some $z\in\RR$ and $G\equiv 0$ in problem (\ref{eq:minFA+G}), i.e., we consider the minimization problem
\begin{align*}
\min_{x\in\RR^{2}}\tfrac12(Ax-z)^{2}.
\end{align*}
We consider the accelerated Chambolle-Pock method (Algorithm 2 in~\cite{chambolle-pock-first-order}) which is (with mismatch) 
\begin{align*}
  x^{i+1} & = x^{i} + \tau_{i}V^{*}y^{i}\\
  \theta_{i} & = \tfrac1{\sqrt{1+2\tau_{i}}},\ \tau_{i+1} = \theta_{i}\tau_{i},\ \sigma_{i+1} = \sigma_{i}/\theta_{i}\\
  y^{i+1} & = \frac{y^{i} + \sigma_{i+1}A(x^{i+1} + \theta_{i}(x^{i+1} - x^{i})) - \sigma_{i+1}z}{1+\sigma_{i}}.
\end{align*}
We initialize the stepsizes with $\tau_{0}\sigma_{0}<1/\norm{A}^{2}$ and the iterates with 
\begin{align*}
x^{0} =
  \begin{pmatrix}
    0\\0
  \end{pmatrix},\quad y^{0} = -z.
\end{align*}
For the mismatch we take $V =
\begin{pmatrix}
  1 & -1
\end{pmatrix}
$.
A standard calculation shows that one gets 
\begin{align*}
  x^{n} & = \sum_{i=0}^{n}\tau_{i}\, V^{*}z,\\
  \text{and}\ y^{n} & = y^{0} = -z.
\end{align*}
The sequence $\tau_{i}$ fulfills
\begin{align*}
\tau_{i+1} & = \tfrac{\tau_{i}}{\sqrt{1+2\tau_{i}}}
\end{align*}
from which we deduce
\begin{align*}
\tfrac1{\tau_{i+1}} = \tfrac{\sqrt{1+2\tau_{i}}}{\tau_{i}}\leq \tfrac{\sqrt{1+2\tau_{i} + \tau_{i}^{2}}}{\tau_{i}} = 1 + \tfrac1{\tau_{i}}.
\end{align*}
Hence, it holds that $1/\tau_{i+1}\leq i+1$, i.e. $\tau_{i}\geq 1/i$ and thus, the iterates $x^{i}$ do diverge if $z\neq 0$.

If we use the non-accelerated variant (Algorithm 1 from~\cite{chambolle-pock-first-order})) we can show divergence of the $x^{i}$ as well.
\end{example}

\section{Numerical examples}
\label{sect:examples}

In this section we report some numerical experiments to illustrate the results.

\subsection{Convex quadratic problems}
\label{sec:conv-quad}

As examples where all quantities and solutions can be computed exactly, we consider convex quadratic problems of the form
\begin{align}\label{eq:quadratic-test-problem}
  \min_{x\in\RR^{n}}\tfrac{\alpha}{2}\norm{x}_{2}^{2} + \tfrac{1}{2\beta}\norm{Ax-z}_{2}^{2}
\end{align}
with $\alpha,\beta>0$, $A\in\RR^{m\times n}$ and $z\in\RR^{m}$. With $G(x) = \tfrac\alpha2\norm{x}_{2}^{2}$ and $F(\zeta) = \tfrac1{2\beta}\norm{\zeta-z}_{2}^{2}$ this is of the form~(\ref{eq:minFA+G}). The conjugate functions are $G^{*}(\xi) = \tfrac1{2\alpha}\norm{x}_{2}^{2}$ and $F^{*}(y) = \tfrac\beta2\norm{y}_{2}^{2} + \scp{y}{z}$ and the respective proximal operators are readily computed as
\begin{align*}
  \prox_{\tau G}(x) &= \tfrac{x}{1+\tau\alpha}\\
  \prox_{\sigma F^{*}}(y) & = \tfrac{y-\sigma z}{1+\sigma\beta}
\end{align*}
and the optimal primal solution is
\begin{align*}
  x^{*} = \Big(\alpha I + \tfrac1\beta A^{T}A\Big)^{-1}(\tfrac1\beta A^{T}z).
\end{align*}
Note that $G$ is strongly convex with constant $\gamma_{G} = \alpha$ and $F^{*}$ is strongly convex with constant $\gamma_{F^{*}} = \beta$ and hence, for $\alpha,\beta>0$ we can use Theorem~\ref{thm:existence_linear_convergence} to obtain valid stepsizes.
For a numerical experiment we choose $n=400$, $m=200$, a random matrix $A\in\RR^{m\times n}$ and a perturbation $V\in\RR^{m\times n}$ by adding a small random matrix to $A$, i.e.
\begin{align*}
  V = A + E\ \text{ with }\ \norm{E}\leq \eta.
\end{align*}
The resulting algorithm is
\begin{align}\label{eq:cp-mm-quadratic}
  \begin{split}
    x^{i+1} & = \tfrac{1}{1+\tau\alpha}(x^{i}-\tau V^{T}y^{i})\\
    y^{i+1} & = \tfrac{1}{1+\sigma\beta}\Big(y^{i} + \sigma A(x^{i+1} + \omega(x^{i+1}-x^{i})) - \sigma z\Big).
  \end{split}
\end{align}
We check the condition $\gamma_{G}\gamma_{F^{*}}>2\norm{A-V}^{2}$ numerically and use Theorem~\ref{thm:existence_linear_convergence} to obtain feasible stepsizes.
For constant stepsizes, we get as limit the unique fixed points
\begin{align}
  \label{eq:cp-mm-quadratic-fp-mm}
  \hat x & = \Big(\alpha I + \tfrac{1}{\beta}V^{T}A\Big)^{-1}(\tfrac1\beta V^{T}z) = V^T(\alpha\beta I + AV^{T})^{-1}z\\
  \hat y & = -(\beta I + \tfrac{1}{\alpha}AV^{T})^{-1}z = -\alpha(\alpha\beta I + AV^{T})^{-1}z\nonumber
\end{align}
while the true primal solution is
\begin{align}
  \label{eq:cp-mm-quadratic-fp}
  x^{*} = \Big(\alpha I + \tfrac{1}{\beta}A^{T}A\Big)^{-1}(\tfrac1\beta A^{T}z)  = A^T(\alpha\beta I + AA^{T})^{-1}z
\end{align}
For our experiment we used $\alpha = \gamma_G=0.15$ and $\beta = \gamma_{F^{*}}=1$ and $\kappa=0.01$ in Theorem~\ref{thm:existence_linear_convergence}.

Figure~\ref{fig:quadratic_linear_conv} illustrates that the method with mismatched adjoint behaves as expected: We observe linear convergence towards the fixed point $\hat x$ and the iterates reach the error to the true minimizer $x^{*}$ that has been predicted by Theorem~\ref{thm:error-estimate}.
\begin{figure}[H]
\centering
\includegraphics[width=5cm]{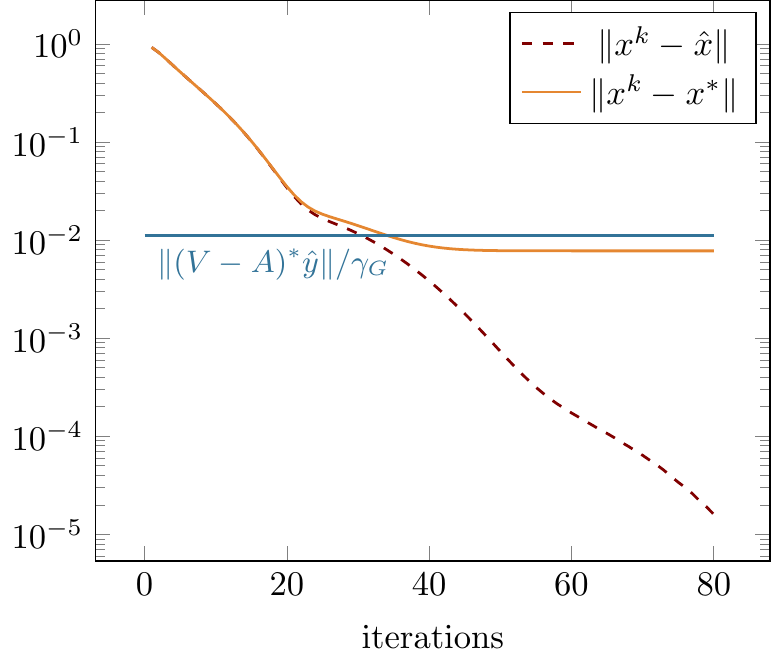} 
\caption{Convergence of iteration~\eqref{eq:cp-mm-quadratic}. Here $\hat x$ is the fixed point~\eqref{eq:cp-mm-quadratic-fp-mm} of the iteration with mismatch and $x^{*}$ is the original primal solution~\eqref{eq:cp-mm-quadratic-fp}. The solid orange plot is the distance of the primal iterates $x^{k}$ of~\eqref{eq:cp-mm-quadratic} to the fixed point of the iteration, and the dashed purple line is the distance of the iterates $x^{k}$ to the original primal solution. As predicted, the latter distance falls below the value given in Theorem~\ref{thm:error-estimate}.}\label{fig:quadratic_linear_conv}
\end{figure}

\subsection{Computerized tomography}
\label{sec:computerized-tomography}
To illustrate a real-world application of our results, we consider the problem of computerized tomography (CT)~\cite{buzug2008computed}. In computerized tomography one aims to reconstruct a slice of an object from x-ray measurements taken in different directions. The x-ray measurements are stored as the so-called sinogram and the map of the image of the slice to the sinogram is modeled by a linear map which is referred to as Radon transform or forward projection. The adjoint of the map is called backprojection. There exist various inversion formulas which express the inverse of the Radon transform explicitly, but since the Radon transform is compact (when modeled as a map between the right function spaces~\cite{natterer2001mathematics}), any inversion formula has to be unstable. One popular stable, approximate inversion method is the so called filtered backprojection (FBP)~\cite{buzug2008computed}. The method gives good approximate reconstruction when the number of projections is high and when the data is not too noisy. However, the quality of the reconstruction quickly gets worse when the number of projections decreases. There are numerous efforts to increase reconstruction quality from only a few projections, as this lowers the x-ray dose for a CT scan. One successful approach uses total variation (TV) regularization as a reconstruction method~\cite{sidky2012convex} and solves the respective minimization problem with the Chambolle-Pock method. Usually, the method takes a large number of iterations. Moreover, there are many ways to implement the forward and the backward projection. In applications it sometimes happens that a pair of forward and backward projections are chosen that are not adjoint to each other, either because this importance of adjointness is not noted, or on purpose (speed of computation, special choice of backprojector to achieve a certain reconstruction quality, see also~\cite{zeng2000unmatched,Palenstijn2011PerformanceIF,Chouzenout2021pgm-adjoint,Lorenz2018TheRK}). 

We describe a discrete image with $m\times n$ pixels as $x\in\RR^{m\times n}$. Its discrete gradient $\nabla x = u\in\RR^{m\times n\times 2}$ is a tensor and for such a tensor we define the pixel-wise absolute value in the pixel $(i_{1},i_{2})$ as $\abs{u}_{i_{1},i_{2}}^2 = \sum_{k=1}^2 u_{i_{1},i_{2},k}^2$, cf.~\cite[p. 416]{Bredies2019MathematicalIP}. For images $x\in\RR^{m\times n}$ we denote by $\norm{x}_{p} = \left( \sum_{i_{i},i_{2}}\abs{x}_{i_{1},i_{2}}^{p} \right)^{1/p}$ the usual pixel-wise $p$-norm.
With $R$ we denote the discretized Radon transform taking an $m\times n$-pixel image to a sinogram of size $s\times t$. 
We aim to solve the problem
\[
\min _{x \in  \mathbb{R}^{m \times n}} \frac{\lambda_0}{2}\|R x-z\|_{2}^{2}+ \frac{\lambda_1}{2} \norm{\abs{\nabla x}}_{1}+ \frac{\lambda_2}{2} \norm{x}_2^2
\]
for a given sinogram $z$ and constants $\lambda_{0},\lambda_1,\lambda_{2}>0$.
This can be expressed as the saddle point problem
\[
\min_{x \in \mathbb{R}^{m \times n}} \max _{\substack{p \in \mathbb{R}^{m \times n \times 2}\\ q \in \mathbb{R}^{s \times t}}}- \langle x, \operatorname{div} p\rangle + \langle R x-z, q\rangle - \frac{1}{2 \lambda_0} \norm{q}^2 - I_{\norm{\cdot}_{\infty} \leq \lambda_1}(p) + \frac{\lambda_2}{2} \norm{x}^2.
\]
With $F^*(q,p) = 
\frac{1}{2 \lambda_0}\|q\|_{2}^{2} + \langle q, z \rangle + I_{\norm{\cdot}_{\infty} \leq \lambda_1}(p)$, $G(x) = \frac{\lambda_2}{2} \norm{x}_2^2$ and $A =
\left(\begin{array}{l}
R \\
\nabla
\end{array}\right)
$ the saddle point formulation is exactly of the form~(\ref{eq:minFA+G}).
The function $G$ is strongly convex, however, $F^{*}$ is not. Hence, we regularize further by adding $\epsilon\|p\|_{2}^{2}/2$ with $\epsilon>0$ to $F^{*}$ which amounts to a Huber-smoothing of the total variation term in the primal problem.

In our experiment we want to recover the Shepp Logan phantom $\hat{x}$ with $400 \times 400$ pixels from measurement with just 40 equispaced angles and 400 bins for each angle and a parallel bean geometry, and added 15\% relative Gaussian noise. Hence, the resulting sinogram $z$ is of the shape $40 \times 400$. To implement the mismatch we used non-adjoined implementations of the forward and back-projection (we used the Astra toolbox~\cite{van2016fast} and took as forward operator $A$ the parallel strip beam projector and as backward projection $V^{*}$ the adjoint of the parallel line beam projector, see \url{https://www.astra-toolbox.com/docs/proj2d.html} for the documentation). All experiments are done in Python 3.7.

We use the algorithm from Theorem \ref{thm:linear_conv_rate} where we used $V^{*} =
\begin{pmatrix}
  S^{*} & -\operatorname{div}
\end{pmatrix}
$ instead of $A^{*}$ and replace the correct adjoint $S^{*}$ by the computationally more efficient  adjoint of the parallel line projector.
To achieve a fair comparison, we vary the regularization parameter $\lambda_1$ of the total variation penalty from $0.6$ to $2.4$. The remaining parameters are set to $\lambda_0 = 1$ and $\lambda_2 = \epsilon = 0.01$.  The initial stepsizes are set according to Theorem \ref{thm:linear_conv_rate} for the mismatched adjoint and  \cite[Algorithm 3]{chambolle-pock-first-order} in the non-mismatched case, respectively.
 The operator norm of the adjoint operator $S^{*}$ is computed numerically and the operator norm of the gradient is estimated like in (cf.~\cite[Lemma 6.142]{Bredies2019MathematicalIP}).

In Figure~\ref{fig:TV-rec} we show the original image (the famous Shepp-Logan phantom), the reconstruction by filtered backprojection and the best results of the Chambolle-Pock iteration for TV regularized reconstruction with the exact and with the mismatched adjoint. One notes that the use of a mismatched adjoint leads to a good reconstruction, comparable to the result with \cite[Algorithm 3]{chambolle-pock-first-order}.
 
\begin{figure}[htb]
\centering\scriptsize
\begin{tabular}{rcccc}
&  Original & Sinogram & \shortstack{Rec. \\ with FBP} & \\
& \includegraphics[width=1.75cm]{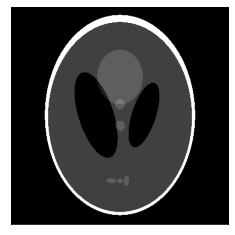} &
 \includegraphics[height=1.75cm, width=1.75cm]{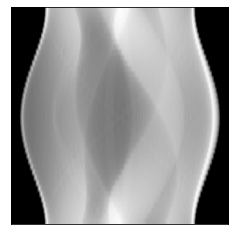} & \includegraphics[width=1.75cm]{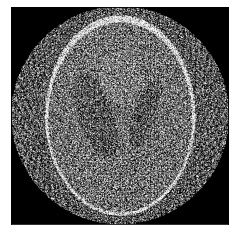} & \\
 & \shortstack{Rec. \\ with V}   &  \shortstack{Rec. Error \\ with V} & \shortstack{Rec.  \\ with A} & \shortstack{Rec. Error  \\ with A} \\
\rotatebox{90}{\quad $\lambda = 0.6$}& 
\includegraphics[height=1.75cm]{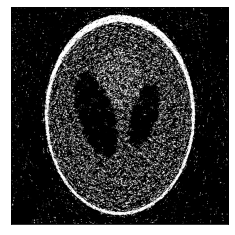}& 
\includegraphics[height=1.75cm]{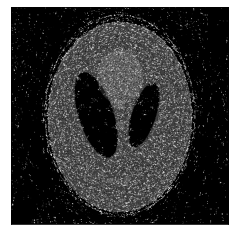} & 
\includegraphics[height=1.75cm]{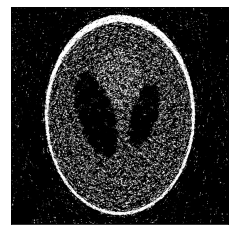} &
\includegraphics[height=1.75cm]{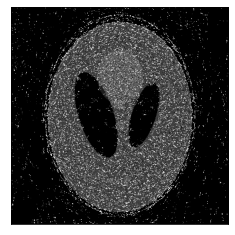} \\
\rotatebox{90}{\quad $\lambda = 1.2$}& 
\includegraphics[height=1.75cm]{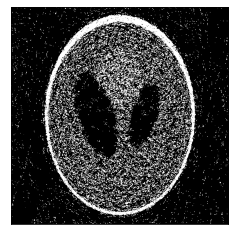}& 
\includegraphics[height=1.75cm]{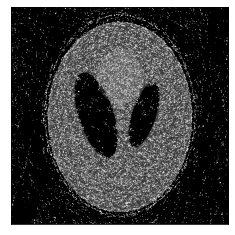} & 
\includegraphics[height=1.75cm]{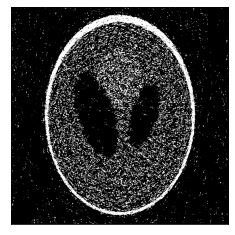} &
\includegraphics[height=1.75cm]{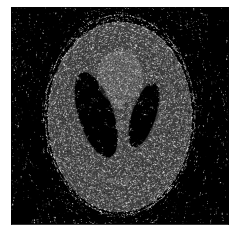}  \\
\rotatebox{90}{\quad $\lambda = 2.4$}& 
\includegraphics[height=1.75cm]{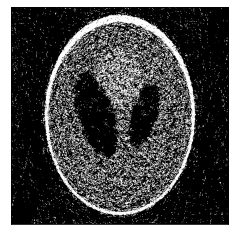}& 
\includegraphics[height=1.75cm]{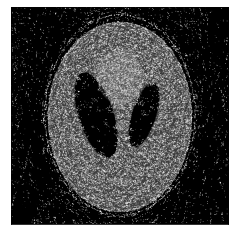} & 
\includegraphics[height=1.75cm]{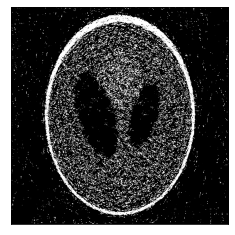} &
\includegraphics[height=1.75cm]{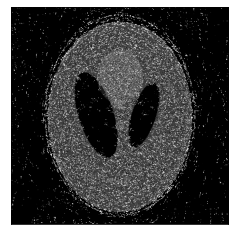} \\
\end{tabular}
\caption{Reconstruction (\textit{Rec.}) of the Shepp Logan phantom. From left to right: Reconstruction with adjoint mismatch with fixed grayscale, the absolute reconstruction error towards the original image, the reconstruction with the exact adjoint and the corresponding absolute error. All images have a fixed grayscale with values reaching from 0.0 to 1.0.}
\label{fig:TV-rec}
\end{figure}

Figure~\ref{fig:TV-errors} shows the distance of the iteration to the exact reconstruction (i.e. the true, noise-free Shepp-Logan phantom). Naturally, the Chambolle-Pock iteration does not drive the error to the noise-free solution to zero (for both exact and mismatched adjoint). There are at least three different reasons: There is some error in the sinogram, we only use few projections and there is TV regularization involved. However, the non-mismatched Chambolle-Pock method gets admittedly closer to the original image than the mismatched method. Figure~\ref{fig:TV-objective} shows the primal objective. We note that in this example the iteration with mismatch yields results comparable to the non-mismatched Chambolle-Pock method. Moreover, it can be seen that, as expected, the use of a mismatched adjoint prevents the true minimization of the objective.
With the computationally more efficient parallel line projector as adjoint, we are able to decreases the computation time significantly with approximately $15\%$ average time saving per iteration on a 2020 M1 MacBook Air running macOS Big Sur. However, the non-mismatched method takes less iterations to retrieve a good result, so the no computational advantage can be shown by this experiment.

\begin{figure}[htb]
\centering
\includegraphics{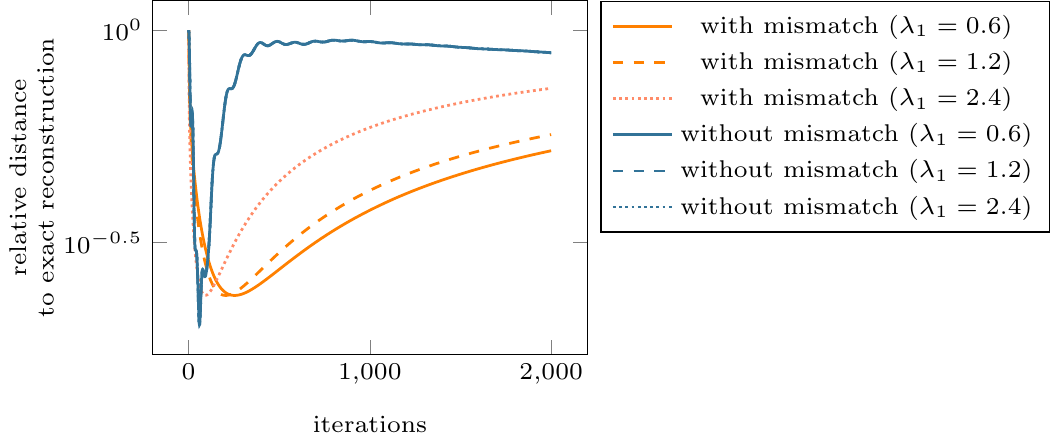} 
\caption{Relative distance to the exact reconstruction over iterations.}
\label{fig:TV-errors}
\end{figure}

\begin{figure}[htb]
\centering
\includegraphics{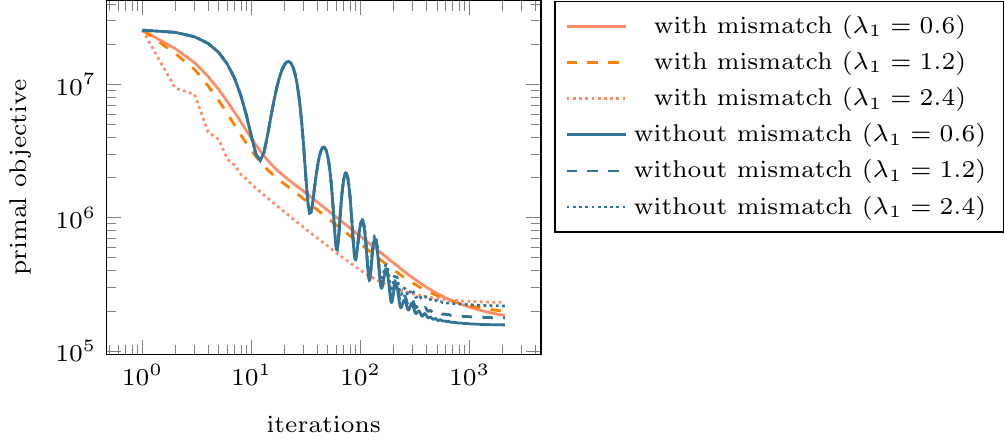} 
\caption{Decay of the primal objective function value.}
\label{fig:TV-objective}
\end{figure}

\section{Conclusion}
\label{sect:conclusion}

We have established stepsizes, for which the Chambolle-Pock method converges, even if the adjoint is mismatched. Additionally, we presented results that showed, that not only strong convergence can be preserved under strong convexity assumptions, and also  that the convergence rate is in a similar region. Furthermore, as a broad class of problems are in the scope of this paper, we established an upper bound on the distance between the original solution and the fixed point of iteration with mismatch. Thus, approximating the adjoint with a computationally more efficient algorithm can be done as long as the assumptions are respected. One of these assumptions is, that the iteration with mismatch still possesses a fixed point and more work is needed to understand when this is guaranteed. Furthermore, we discussed advantages and disadvantages of our analysis and illustrated our results on an example from computerized tomography, in which the mismatched adjoint was obtained using different discretization schemes for the forward operator and the mismatched adjoint.

\section*{Acknowledgement}

This work has received funding from the European Union's Framework Programme for Research and Innovation Horizon 2020 (2014-2020) under the Marie Sk\l odowska-Curie Grant Agreement No. 861137. \raisebox{-0.25ex}{\includegraphics[height=2ex]{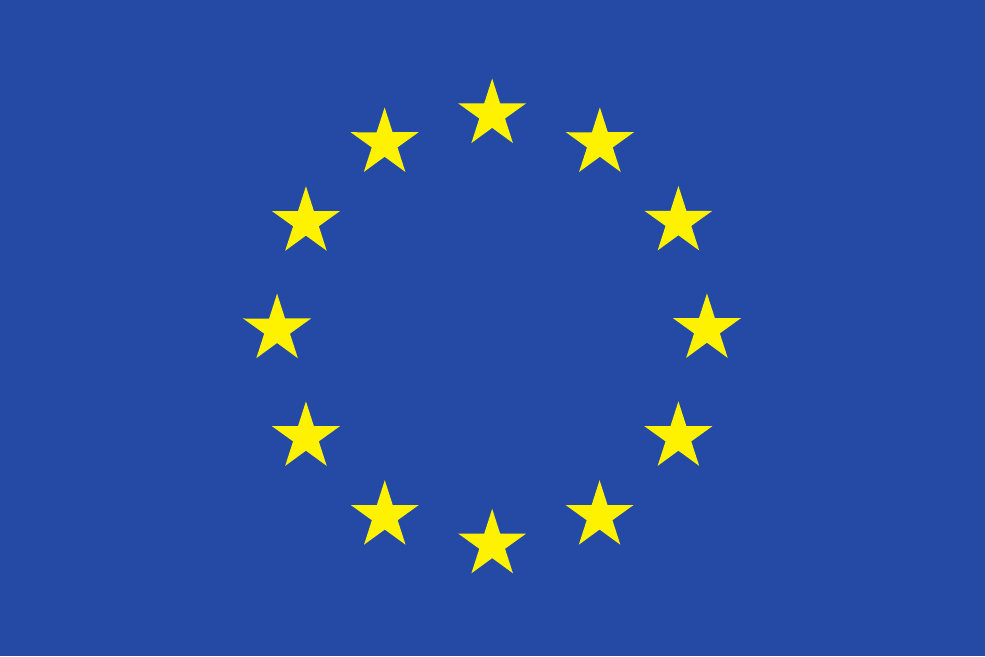}}

\subsection*{Compliance with Ethical Standards}
\label{sec:compliance}
The authors declare that there is no conflict of interest.

\bibliography{references.bib}

\end{document}